\documentclass[12pt]{article}
\usepackage{amsfonts}
\usepackage{amssymb}
\usepackage{newlfont}
\usepackage{amsthm}
\usepackage{euscript}
\usepackage{amsmath}
\usepackage{graphicx}
\usepackage{dsfont}
\usepackage{enumerate}
\usepackage{amsmath}
\usepackage[]{color}
\usepackage{graphicx}
\newcommand{\sgn}{\mathop \mathrm{sgn}}

\newtheorem{Proposition}{Proposition}
\theoremstyle{definition}

\theoremstyle{plain}
\newtheorem{Lemma}{Lemma}
\newtheorem{Theorem}{Theorem}

\theoremstyle{plain}

\theoremstyle{remark}

\newcommand{\vp}{X}

\newcommand{\R}{\mathbb{R}}

\newcommand{\E}{\mathbb{E}}
\begin{document}
\begin{titlepage}
\begin{center}
\large {\bf Olga Aryasova}
\end{center}
Institute of Geophysics, National Academy of Sciences of Ukraine,
Palladin pr. 32, 03142, Kyiv, Ukraine; National Technical University of Ukraine "Igor Sikorsky Kyiv Politechnic Institute", Kyiv, Ukraine, oaryasova@gmail.com, ORCiD ID  0000-0002-5870-8349.
\vskip 10 pt
\begin{center}
\large {\bf Andrey Pilipenko}
\end{center}
Institute of Mathematics,  National Academy of Sciences of
Ukraine, Tere\-shchen\-kivska str. 3, 01004, Kyiv, Ukraine; National Technical University of Ukraine "Igor Sikorsky Kyiv Politechnic Institute", Kyiv, Ukraine,\\ pilipenko.ay@gmail.com.
\vskip 10 pt
\begin{center}
\large {\bf Sylvie Roelly}
\end{center}
Institut f\"ur  Mathematik,  Universit\"at Potsdam, Karl-Liebknecht-Str. 24-25, D-14476, Potsdam OT Golm, Germany, roelly@math.uni-potsdam.de.
\vskip 20 pt

{\bf Funding}
\vskip 10 pt

{This work was partially supported by DFG project "Stochastic Dynamics with Interfaces", no. 452119141, and the Alexander von Humboldt Foundation (Research Group Linkage cooperation {\it Singular diffusions: analytic and stochastic approaches}) between the University of Potsdam and the Institute of Mathematics of the National Academy of Sciences of Ukraine.}
\end{titlepage}

\begin{center}
{\Large \bf Exponential a.s. synchronization of one-dimensional diffusions with non-regular coefficients}
\end{center}
\begin{abstract} We study the asymptotic behaviour of a real-valued diffusion whose non-regular drift is given as a sum of a dissipative term and a bounded measurable one. We prove that
two trajectories of that diffusion converge a.s. to one another  at an exponential explicit rate as soon as the dissipative coefficient is large enough. A similar result in $L_p$ is obtained.
\end{abstract}
\vskip 20 pt

\section{Introduction}\label{Sec_Introduction}
The main object of our study is a one-dimensional stochastic differential equation (SDE) of the type
\begin{equation*}\label{eq_main0}
\left\{
\begin{aligned}
d\vp_t&=\left(-\lambda\vp_t+a(\vp_t)\right)dt+\sigma(\vp_t) \, dw_t, \ t>0,\\
\vp_0&\equiv x,\\
\end{aligned}\right.
\end{equation*}
where $\lambda$ is a positive real number, the drift $a$ is \emph{measurable}, the diffusion coefficient $\sigma$ is a Lipschitz continuous non-degenerate function, and $(w_t)_{t\geq0}$ is a Wiener process.

Thanks to the celebrated transform method, Zvonkin proved in \cite{Zvonkin74} that this SDE admits a unique strong solution, which we will denote by $(X_t^x)_{t>0}$.
Moreover, it was proved  during the last decade that due to the presence of noise, the family of processes  $(X_t^x)_{t\geq 0, x\in \R}$ shows good spatial regularity properties even if the drift function is discontinuous, see for example
\cite{Aryasova+12, Banos+2017, Bogachev+15, Catellier+2016, Fedrizzi+13b, Flandoli+10, MeyerBrandis+10, Mohammed+12, Hudde+19}.

Concerning the time asymptotic stability of the process $(X_t^x)_{t\geq 0}$ there are much less results in the literature. In case $\lambda$ is large enough, which corresponds to a strong attraction of the dynamics towards 0 and a strong dissipativity, it is natural to expect that, asymptotically in time, $X_t^x$ will forget its initial position $x$. Indeed, under Lipschitz continuity assumption on the drift function $a$, it is proved e.g. in \cite{Ito+64}, that the $L_p$-distance between $X_t^x$ and $X_t^y, y\not = x$, vanishes as $t$ tends to $+\infty$, but no rate is available. In \cite{Flandoli+2017} the stabilisation is shown as a convergence in probability of $X_t^x-X_t^y$ towards 0, under ${\mathcal C}^1$-regularity assumption on the drift function via the negativity of the associated top Lyapunov exponent.
For diffusions whose drift function is not differentiable but admits a finite variation, an explicit representation of the Sobolev derivative of $x \mapsto X_t^x$ can be found in \cite{Aryasova+12}. This representation makes it possible to find an exponential decreasing rate for $|X_t^x - X_t^y|, y\not = x$ as $t\to\infty$, when a stationary distribution exists.
Recently, such   asymptotic stability was obtained in a multidimensional framework, for diffusions whose drift function admits jump discontinuities concentrated along a hyperplane, see \cite{Aryasova+2019}.

In the present paper, we address  and solve the question of  \emph{almost sure synchronization} - see the exact definition in \eqref{eq_stability_dissip} - in high dissipative regime ($\lambda$ large) for a wide class of SDEs with {\it irregular drift} functions: the function $a$ is only supposed to be the sum of a Lipschitz function and of a bounded {\it  measurable} one. Furthermore, we exhibit an explicit {\it exponential convergence rate} to 0 for $|X_t^x- X_t^y|$, both {\it almost surely} (see \eqref{eq:clambda} and \eqref{eq:clambda3'}) and in $L_p$. To our knowledge it is the first result of that type under such general assumptions.

Note that in absence of noise ($\sigma \equiv 0$), there is no reason to expect synchronization of $(X_t^x)_{t\geq 0, x\in \R}$ asymptotically in time. Indeed, consider the ODE \mbox{$u'(t)=-\lambda u(t) + \ \sgn(u(t))$} with initial condition $x $, whose unique solution is given by \mbox{$u(t)= \frac{\sgn (x)}{\lambda} + \left(x - \frac{\sgn (x)}{\lambda}\right) e^{-\lambda t}$}. Thus
$
\lim_{t\to\infty}u(t)=\frac{\sgn (x)}{\lambda}
$
which exhibits a clear discontinuity at the point $x=0$, which corresponds in fact to the  (unique) discontinuity point of the drift function $a = \sgn$. We are thus in presence of a phenomena known in the literature as {\it synchronization by noise}, see \cite{Flandoli+2017}.

In the spirit of Zvonkin, our approach is based on an accurately chosen space-transform
in such a way that the transformed SDE - written via the new coordinate - has a simpler structure.
A similar method could theoretically be used in more general context - multidimensional diffusions or  SDEs with L\'evy-noise. However, the construction of corresponding transforms requires the investigation of elliptic equations whose solution is a non-trivial problem.

The paper is organized as follows. The main results are formulated in Section \ref{Sec_Main} and the proofs are presented in Section \ref{Sec_Proof}.

\section{Main results}\label{Sec_Main}
First we study the asymptotic behavior with respect to its initial condition of the strong solution of an SDE with regular dissipative drift term. Though the result seems to be well known, we failed to find an exact reference. Besides, the proof is instructive itself.
\begin{Proposition}\label{Prop1}
Consider the SDE
\begin{equation}\label{eq_dissip}
dY_t=b(Y_t)dt+\sigma(Y_t) \,dw_t, \ t>0,
\end{equation}
where $(w_t)_{t\geq0}$ is a Wiener process.
Suppose that the following assumptions hold:
\begin{enumerate}
\item [($\mathcal{H}_1$)]
The drift $b$ is 
continuous and  satisfies a dissipative condition:
$$
\exists D_b>0 \quad \forall x,y\in \R \quad \left(b(y)-b(x)\right)(y-x)\leq - D_b(y-x)^2;
$$
\item [($\mathcal{H}_2$)]  The diffusion coefficient $\sigma$ is a global Lipschitz continuous function :
\begin{equation*}
    \exists L_\sigma>0 \quad \forall x,y\in \R \quad    |\sigma(y)-\sigma(x)|\leq L_\sigma|y-x|.
        \end{equation*}
\end{enumerate}
Then, denoting by $(Y_t^x)_{t\geq 0}$ the unique strong solution of \eqref{eq_dissip} starting in $x \in \R$, the following \emph{almost sure synchronization} at exponential rate holds: for any $ c<D_b$,
\begin{equation}\label{eq_stability_dissip}
\forall x,y\in \R, \quad  \quad \lim_{t\to +\infty}|Y_t^y-Y_t^x| \, e^{c t}= 0 \quad \mbox{ a.s. }
\end{equation}
Moreover, if $c_{p,b,\sigma}:= D_b - \frac{p-1}{2}L_\sigma^2$ is positive, the following bound holds in $L_p, p\geq2 $:
\begin{equation}\label{eq_stability_dissip_Lp}
\forall x,y\in \R, \quad \forall t>0,  \quad \parallel Y_t^y-Y_t^x\parallel_p \leq |y-x| \, e^{-c_{p,b,\sigma} t}.
\end{equation}
\end{Proposition}

\vspace{5mm}

The main result of the paper, which follows, concerns the asymptotic behavior of the solution of an SDE generalising \eqref{eq_dissip}, whose drift function $b$ is the sum of a linear dissipative term, a globally Lipschitz term $\beta$ and a {\it non regular} bounded term $\alpha$.
\begin{Theorem}\label{Theorem_main}
Consider the SDE
\begin{equation}\label{eq_main}
\left\{
\begin{aligned}
d\vp_t&=\left(-\lambda\vp_t+\beta (\vp_t) +\alpha(\vp_t)\right) \, dt+\sigma(\vp_t) \, dw_t, \ t>0,\\
\vp_0&=x,\\
\end{aligned}\right.
\end{equation}
where $(w_t)_{t\geq0}$ is a Wiener process.
Suppose that the following assumptions hold:
\begin{enumerate}
\item [($\mathcal{A}_1$)]   The function $\beta$ is global Lipschitz continuous:
\begin{equation*}\label{eq_beta_Lip}
    \exists L_\beta\geq 0\quad \forall x,y\in \R \quad  |\beta(y)-\beta(x)|\leq L_\beta|y-x|;
        \end{equation*}
\item [($\mathcal{A}_2$)] The function $\sigma $ is global Lipschitz continuous:
 \begin{equation*}
         \exists L_\sigma\geq0 \quad \forall x,y\in \R \quad
|\sigma(y)-\sigma(x)|\leq L_\sigma|y-x|
        \end{equation*}
and it is uniformly elliptic:
        \begin{equation*}
      \exists c_\sigma>0 \quad \forall x\in \R  \quad \sigma^2(x)\geq c_\sigma.
        \end{equation*}
\end{enumerate}
Assume also that one of the following two conditions is satisfied:
\begin{enumerate}
\item [($\mathcal{A}_3$)] the function $\alpha $ is bounded measurable with compact support

or

\item [($\mathcal{A}_3'$)] the function $\alpha $ is measurable and its absolute value is a.s. bounded by a  bounded global Lipschitz function $g\in L_1(\R)$; moreover the functions $\beta$ and $ \sigma$  are supposed to be bounded too.
\end{enumerate}

Then, in high dissipative regime - $\lambda$ large enough - the strong solutions
 of \eqref{eq_main} $X_t^x$ and $X_t^y$ starting at different positions
$x$ and $y$ \emph{almost sure synchronize} at exponential rate,
i.e., there exists $\lambda_0$ such that for any $\lambda>\lambda_0$ there exists
a positive constant $c_\lambda$ given explicitly in \eqref{eq:synch_lambda0},
\eqref{eq:clambda}, \eqref{eq:clambda3'} such that
\begin{equation}\label{eq_theorem_stabil}
 \forall x,y\in \R \quad \lim_{t\to\infty}|X_t^y- X_t^x| \, e^{c_\lambda t}=  0 \quad \mbox{a.s.} \, .
\end{equation}
Moreover, the following bound holds in $L_p, p\geq2 $:
\begin{equation}\label{eq_theorem_ineq}
 \exists C>0, \,c_{\lambda,p}>0 \quad \forall x,y \in \R, \quad \forall t\geq0, \quad
\parallel X_t^y- X_t^x \parallel_p \leq C \, |y-x| \, e^{-c_{\lambda,p} t}.
\end{equation}
\end{Theorem}

\section{Proofs}\label{Sec_Proof}
\begin{proof}[Proof of Proposition \ref{Prop1}]
Notice first that, applying \cite[Proposition 2.1]{Scheutzow+2017}, assumptions ($\mathcal{H}_1$)-($\mathcal{H}_2$) provide the existence of a unique global strong
 solution to \eqref{eq_dissip}, denoted here by $(Y_t^x)_{t\geq 0}$.

Fix any $y>x$ and define the stopping time $\tau:=\inf\{t\geq 0\ | \ Y_t^x   =Y_t^y\}$.
Due to the continuity of
trajectories, one has for $t\in[0,\tau), \ Y_t^y   >Y_t^x$ a.s.

So, applying  It\^o's formula to $\ln(Y_t^y-Y_t^x)$ we obtain
\begin{multline}
d\left(\ln(Y_t^y-Y_t^x)\right)=\\
\left(\frac{b(Y_t^y)-b(Y_t^x)}{Y_t^y-Y_t^x}-
\frac{\left(\sigma(Y_t^y)-\sigma(Y_t^x)\right)^2}{2(Y_t^y-Y_t^x)^2}\right)dt+
\frac{\sigma(Y_t^y)-\sigma(Y_t^x)}{Y_t^y-Y_t^x} \, dw_t,\ \ t\in[0,\tau).
\end{multline}

Moreover, according to \cite[Proposition 2.1]{Scheutzow+2017} the processes $(Y_t^x)_{t\geq 0, x\in \R}$ satisfy a coalescence property which means that as soon as two solutions
meet in a point then they stay together forever: $Y_t^x =Y_t^y, t\geq \tau$ a.s.

Therefore, by ($\mathcal{H}_1$), the following inequality holds at any time:
\begin{equation} \label{eq:vp_dist}
\ln(Y_t^y-Y_t^x)\leq
\ln(y-x)-   D_b t + \int_0^t\frac{\sigma(Y_s^y)-\sigma(Y_s^x)}{Y_s^y-Y_s^x}  1_{s<\tau} \, dw_s
\ \ \mbox{a.s.}
 \end{equation}
where $\ln 0:=-\infty.$
Note that the expression under the integral sign is bounded because the function $\sigma$ is
Lipschitz continuous.

Further, the martingale
$\displaystyle \int_0^t \frac{\sigma(Y_s^y)-\sigma(Y_s^x)}{Y_s^y-Y_s^x} 1_{s<\tau} \, dw_s$
can be represented as a Brownian motion computed at the random time
 $\displaystyle \int_0^t\left(\frac{\sigma(Y_s^y)-\sigma(Y_s^x)}{Y_s^y-Y_s^x}\right)^2 1_{s<\tau} \ ds$
which is uniformly bounded by $L_\sigma^2 \, t$. Thus the law of iterated logarithm  yields
\begin{equation} \label{eq:vp_dist_2}
\lim_{t\to\infty} \frac{1}{t} \int_0^t\frac{\sigma(Y_s^y)
-\sigma(Y_s^x)}{Y_s^y- Y_s^x} 1_{s<\tau} \,dw_s =0 \ \ \ \mbox{a.s.}
\end{equation}
Now the decreasing rate of $|Y_t^y- Y_t^x|$ announced in \eqref{eq_stability_dissip} follows from \eqref{eq:vp_dist} and \eqref{eq:vp_dist_2}.

Let us now prove the $L_p$-bound, $p\geq 2$.
For any constant $k$, it follows from It\^o's formula
\begin{multline*}
|Y_t^y-Y_t^x|^p e^{k t}= |y-x|^p\\
    + \int_0^t \Big(k|Y_s^y-Y_s^x|^2+p\left(b(Y_s^y)-b(Y_s^x)\right)(Y_s^y-Y_s^x)\Big.
  \\ +\left.\frac{p(p-1)}{2}\left(\sigma\left(Y_s^y\right)-\sigma(Y_s^x)\right)^2\right)|Y_s^y-Y_s^x|^{p-2} e^{k s}ds\\
 + \int_0^t p\left(\sigma(Y_s^y)-\sigma(Y_s^x)\right)\sgn(Y_s^y-Y_s^x)|Y_s^y-Y_s^x|^{p-1}e^{k s}\, dw_s.
\end{multline*}
Using the dissipativity of $b$ and the Lipschitzianity of $\sigma$ we get
\begin{multline} \label{ineq:Lpbound}
  |Y_t^y-Y_t^x|^p e^{k t}\leq |y-x|^p \\
  + \int_0^t \left(k-p D_b+\frac{p(p-1)}{2}L_\sigma^2\right)|Y_s^y-Y_s^x|^p e^{k s}\, ds\\
  + \int_0^t p\left(\sigma(Y_t^y)-\sigma(Y_s^x)\right)\sgn(Y_s^y-Y_s^x)|Y_s^y-Y_s^x|^{p-1}e^{k s}dw_s.
\end{multline}
The stochastic Gronwall lemma proved in  \cite{Scheutzow2013} allows to deduce that - see also the stochastic Gronwall-Lyapunov inequality in \cite{Hudde+19} which improves considerably the former results - 
 $$
 \forall T\geq0 \quad \forall x,y \in\R \quad \forall p\geq 2: \sup_{t\in[0,T]}\E|Y_t^y-Y_t^x|^p< +\infty .
$$
%
So,  since $\sigma$ is Lipschitz continuous,
the stochastic integral in the rhs of \eqref{ineq:Lpbound} is not only a local martingale but also an integrable centered martingale.
Now, as soon as $\displaystyle k\leq p D_b-\frac{p(p-1)}{2}L_\sigma^2$,
$$\E\left(|Y_t^y-Y_t^x|^p\right) e^{k t}\leq |y-x|^p $$
which implies \eqref{eq_stability_dissip_Lp}.

\end{proof}


\begin{proof}[Proof of Theorem \ref{Theorem_main}]
Notice first that assumptions ($\mathcal{A}_1$),($\mathcal{A}_2$) and the boundedness of $\alpha$ provide the existence of  a unique strong solution to equation \eqref{eq_main}. This result follows from \cite{Zvonkin74} via a localization method. \\

Now, since the function $\alpha$ appearing in the drift is not regular we cannot apply directly Proposition \ref{Prop1}. Our first step will then consist to follow Zvonkin's idea and transform the dynamics of \eqref{eq_main} into an SDE with regular drift. Unfortunately by removing only the irregular term $\alpha$, we do not obtain a transformed dynamics satisfying the dissipative assumption ($\mathcal{H}_1$). We then introduce a bounded, Lipschitz continuous, integrable \emph{intermediate function} $\gamma$, whose exact choice will be done later, see \eqref{eq:gamma3} and \eqref{eq:gamma3'}. A partial Zvonkin's transform to remove the drift $\alpha-\gamma$ will yield the SDE \eqref{eq:tildeX}, whose drift $\tilde b:= -\lambda \, \tilde{id}+\tilde \beta+\tilde\gamma$  is indeed  dissipative for $\lambda$ large enough, as we will prove.\\

So we rewrite equation \eqref{eq_main}  as follows:
\begin{equation*}
\left\{
\begin{aligned}
d\vp_t&=\left(-\lambda\vp_t+ \left(\beta(\vp_t)+\gamma(\vp_t)\right) + \left(\alpha(\vp_t)-\gamma(\vp_t)\right)\right) \,dt + \sigma(\vp_t) \, dw_t, t\geq0,\\
\vp_0&\equiv x.
\end{aligned}\right.
\end{equation*}
To eliminate the non-regular term $\alpha-\gamma$, we define the (partial) scale function $s$ on $\R$ by
\begin{equation}\label{eq_scale}
  s(x):=\int_0^x\exp\left(-2\int_0^y\frac{\alpha(z)-\gamma(z)}{\sigma^2(z)}dz\right)dy, \ x\in\R.
\end{equation}
It is differentiable and
\begin{equation}\label{eq_scale_first_deriv}
  s'(x)=\exp\left(-2\int_0^x\frac{\alpha(z)-\gamma(z)}{\sigma^2(z)}dz\right)
\end{equation}
which is uniformly bounded from below and above as follows:
\begin{equation}\label{eq_scale_deriv_bound}
 0< \frac{1}{L_s}\leq s'(x)\leq L_s< +\infty \quad \textrm{ where } \quad
 L_s:=\exp\left( 2\int_{-\infty}^\infty\frac{|\alpha(z)-\gamma(z)|}{\sigma^2(z)}dz\right).
\end{equation}

The finiteness (resp. positivity) of $ L_s$ is due to the integrability of both $\alpha$ and $\gamma$, combined with the uniform lower bound of $\sigma$.\\
Moreover, the second derivative of $s$ exists for almost all $x$ and satisfies
\begin{equation}\label{eq_scale_second_deriv}
  s''(x)=2\frac{\gamma(x)-\alpha(x)}{\sigma^2(x)} s'(x).
\end{equation}
Due to \eqref{eq_scale_deriv_bound}, $s$ is a \emph{bilateral} Lipschitz continuous function:
\begin{equation}\label{eq_scale_Lip}
\forall x,y\in\R, \quad  \frac{1}{L_s}|y-x|\leq |s(y)-s(x)|\leq L_s|y-x|.
\end{equation}
Since \eqref{eq_scale_second_deriv} yields a uniform bound on $s''$, we get that $s'$ is also global Lipschitz continuous:
\begin{equation}\label{eq_scale_deriv_Lip}
 \forall x,y \in \R, \quad  |s'(y)-s'(x)|\leq L_{s'}|y-x| \text{ where }
L_{s'}:=2\frac{\|\gamma-\alpha\|_\infty}{c_\sigma}L_s.
\end{equation}

The derivative of $s$ being positive, the function $s$ is strictly increasing. Moreover, since $s(\R)=\R$, it admits an inverse function $s^{-1} $ defined on $\R$ and being a \emph{bilateral Lipschitz continuous} function  too:
\begin{equation}\label{eq_scale_inverse_Lip}
\forall x,y\in\R, \quad   \frac{1}{L_s}\, |y-x|\leq |s^{-1}(y)-s^{-1}(x)|\leq L_s\, |y-x|.
\end{equation}
The process $s(\vp_t^x)$ satisfies the following It\^o's formula:
\begin{eqnarray*}\label{eq_s_Ito}
ds(\vp_t^x)&=&s'(\vp_t^x)d\vp_t^x+\frac12 s''(\vp_t^x)\sigma^2(\vp_t^x)dt \nonumber\\
&=& s'(\vp_t^x)\left(-\lambda\vp_t^x+\beta(\vp_t^x)+\gamma(\vp_t^x)\right) dt+s'(\vp_t^x)\sigma(\vp_t^x) \, dw_t.
\end{eqnarray*}
Note that $s''$ may not exist on a negligible set.
However, the applicability of It\^o's formula is justified, see e.g. \cite[Theorem 3]{Zvonkin74}. \\
Denote the process $s(\vp_t^x)$ by $\tilde X_t^x$. It solves the SDE:
\begin{equation} \label{eq:tildeX}
\left\{
\begin{aligned}
d\tilde X_t(x) &=\left(-\lambda \, \tilde{id}(\tilde X_t)+\tilde \beta(\tilde X_t)+\tilde\gamma(\tilde X_t)\right)dt+
\tilde \sigma(\tilde X_t) \,dw_t, \, t>0,\\
\tilde X_0&\equiv s(x),\\
\end{aligned}
\right.
\end{equation}
where the coefficients are given by
$$
\tilde{id}:= s'\circ s^{-1} \cdot s^{-1},
  \tilde\beta:= s'\circ s^{-1} \cdot \beta \circ s^{-1},
  \tilde\gamma:= s' \circ s^{-1}\cdot \gamma \circ s^{-1},
  \tilde\sigma:=s' \circ s^{-1}\cdot \sigma \circ s^{-1}.
$$
We underline that the irregular drift term $\alpha $ disappeared from the dynamics.\\

Next step in the proof of the theorem is to check that, for $\lambda$ large enough, the new drift
$$
\tilde b:= -\lambda \, \tilde{id}+\tilde \beta+\tilde\gamma
$$
 appearing in the transformed SDE \eqref{eq:tildeX} satisfies assumption ($\mathcal{H}_1$) in order to apply Proposition \ref{Prop1} to the process $\tilde X_t$.

\vspace{2mm}
\emph{Regularity of the three terms composing the drift $\tilde b$. }\\
The next lemma is straightforward.
\begin{Lemma}\label{lem:Lip}
If $f$ and $g:\R\to\R$ are two Lipschitz continuous functions with respective constant $L_f$ and $ L_g$, their composition $f\circ g $ is also a continuous  Lipschitz  function with constant $L_f L_g.$
If additionally $f$ and $g$ are bounded, then the product $fg$ is a   Lipschitz continuous
 function too with constant $\|f\|_\infty L_g+\|g\|_\infty L_f.$
\end{Lemma}
It follows from \eqref{eq_scale_inverse_Lip} and  Lemma \ref{lem:Lip} that the functions
$s'\circ s^{-1}, \beta\circ s^{-1}, \gamma\circ s^{-1},  \sigma\circ s^{-1}$
are Lipschitz continuous, with respective Lipschitz constants
$\displaystyle L_{s'}L_s$, $L_\beta L_s$, $L_\gamma L_s$,  $L_\sigma L_s$.
Then the function  $\tilde{id}$ appearing as first term in $\tilde b$ is locally  Lipschitz continuous.\\ 
Since the function  $\gamma$ we will construct will be bounded and Lipschitz continuous, by Lemma \ref{lem:Lip} the function $\tilde\gamma$ is Lipschitz continuous with constant
\begin{equation}\label{eq_tilde_gamma}
  L_{\tilde\gamma}= ( L_sL_\gamma+\|\gamma\|_\infty L_{s'})L_s.
\end{equation}
Let us
now construct the function $\gamma$ such that $\tilde \beta$  and $\tilde \sigma$
are global Lipschitz continuous.
 We distinguish both cases, depending on the assumption satisfied by the measurable function $\alpha$.

- Assumption ($\mathcal{A}_3'$) holds, i.e. $\beta$ and $\sigma$ are bounded. \\
Then, by Lemma \ref{lem:Lip}, $\tilde\beta$ and $\tilde\sigma$ are Lipschitz continuous functions with respective constants
\begin{equation}\label{eq_tilde_sigma}
 L_{\tilde\beta}=(L_sL_\beta+\|\beta\|_{\infty}L_{s'})L_s \quad \text{ and } \quad
L_{\tilde\sigma}=(L_sL_\sigma+\|\sigma\|_\infty L_{s'}) L_s.
\end{equation}

- Assumption ($\mathcal{A}_3$) holds, i.e. $\alpha$ has compact support, denoted by $[-N_{\alpha}, N_\alpha]$. \\
Since $\beta$ and $\sigma$ are not a priori bounded, one can not directly apply Lemma \ref{lem:Lip} to obtain the regularity of $\tilde\beta$ and $\tilde\sigma$.  It will be possible to  construct $\gamma$ with compact support included in $[-N_{\alpha}-1, N_\alpha +1]$.
Since the function $x \mapsto s(x)$ is then linear for $|x|\geq N_\alpha+1$, by checking the increments of $\tilde\beta$ (resp.  $\tilde\sigma$) separately on the intervals $(-\infty,s(-N_{\alpha}-1)] $, $[s(-N_{\alpha}-1), s(N_\alpha +1)]$ and $[ s(N_\alpha +1), +\infty)$ one gets that $\tilde\beta$ and $\tilde\sigma$ are global Lipschitz continuous with respective constant
\begin{equation}\label{eq_tilde_beta}
L_{\tilde\beta}=(L_sL_\beta+\|\beta\|_{N_\alpha+1}L_{s'})L_s \quad
\text{ and }
\quad L_{\tilde\sigma}=(L_sL_\sigma+\|\sigma\|_{N_\alpha+1} L_{s'})L_s,
\end{equation}
 where the following notation is used: $\|f\|_{N_\alpha+1}:=\sup_{|x|\leq N_\alpha+1} |f(x)|$.\\
Notice that all the  above Lipschitz constants $L_{\tilde \beta},  L_{\tilde \gamma},
L_{\tilde \sigma}$ may depend on the intermediate drift function $\gamma$  but not on the real coefficient $\lambda$.\\

\vspace{1mm}
\emph{Dissipative property of the drift $\tilde b$ for $\lambda$ large enough}:\\
We now show that for $\lambda$ large enough, the function $\tilde b= -\lambda \, \tilde{id}+\tilde \beta+\tilde\gamma$ is dissipative and compute its dissipative constant denoted by $D_{\tilde b}$.
To this aim, we will prove that the slope of the function $\tilde{id}$ is bounded from below by $1/2$:
\begin{equation} \label{eq:slopetildeid}
 \forall x,y \in \R, \quad \frac{\tilde{id}(y)- \tilde{id}(x)}{y-x} \geq \frac{1}{2}.
\end{equation}
With other words $\tilde{id}$ satisfies a one-sided Lipschitz property. As soon as \eqref{eq:slopetildeid} is proved, it is straightforward to deduce that
\begin{equation} \label{eq:Dtildeid}
D_{\tilde b} \geq \frac{\lambda}{2} - L_{\tilde \beta} - L_{\tilde \gamma}.
\end{equation}
So, for any $\displaystyle \lambda > 2(L_{\tilde \beta} + L_{\tilde \gamma})$, the drift $\tilde b$ is dissipative.

Let us now construct a bounded, Lipschitz continuous, integrable intermediate function $\gamma$ in such a way that \eqref{eq:slopetildeid} holds true.
It is enough to prove that the derivative of $\tilde{id}(= s'\circ s^{-1} \cdot s^{-1})$, which exists almost everywhere, is bounded from below by $\frac{1}{2}$. In fact, for a.a. $x$,
$$
(\tilde{id})'(x)=\frac{s''\circ s^{-1}(x)}{s'\circ s^{-1}(x)}s^{-1}(x)+s'\circ s^{-1}(x)\frac{1}{s'\circ s^{-1}(x)}
=\left.\left(\frac{s''(u)}{s'(u)}u+1\right)\right|_{u=s^{-1}(x)}.
$$
Recall that, since $s'$ is
an absolutely continuous function, $s''$ exists almost everywhere on $\R$.
It follows from \eqref{eq_scale_Lip} that mappings $s$ and $s^{-1}$ push sets of Lebesgue measure zero
to  sets of Lebesgue measure zero.
Thus ${s''(s^{-1}(x))}$ is independent of  a modification of $s''$ on a negligible set.\\
Taking into account \eqref{eq_scale_second_deriv}, we get
$$
\frac{s''(u)}{s'(u)}u+1=2 \frac{\gamma(u)-\alpha(u)}{\sigma^2(u)}u+1 \quad \textrm{ for a.a. } u.
$$
Let us separate both cases $\mathcal{A}_3$ and  $\mathcal{A}_3'$.

- If assumption ($\mathcal{A}_3$) holds, we denote the compact
 support of the function $\alpha$  as above by  $[-N_{\alpha}, N_\alpha]$.
Fix  a positive number $\delta< \|\alpha\|_\infty$ and define an \emph{odd} function
$\gamma$ as follows (see Figure \ref{Fig. 1}):
\begin{equation} \label{eq:gamma3}
\gamma (u)=\left\{
\begin{aligned}
&\|\alpha\|_\infty \frac{u}{\delta},& \ & u\in[0,\delta],\\
&\|\alpha\|_\infty,&  \ & u\in[\delta,N_\alpha],\\
&\|\alpha\|_\infty (N_\alpha + 1 - u),& \ & u\in[N_\alpha,N_\alpha+1],\\
&0,& \ & u\in[N_\alpha+1, +\infty),\\
&-\gamma(-u),& \ & u\in\R_-.
\end{aligned}\right.
\end{equation}
\begin{figure}[h]
 \includegraphics[width=11 cm]{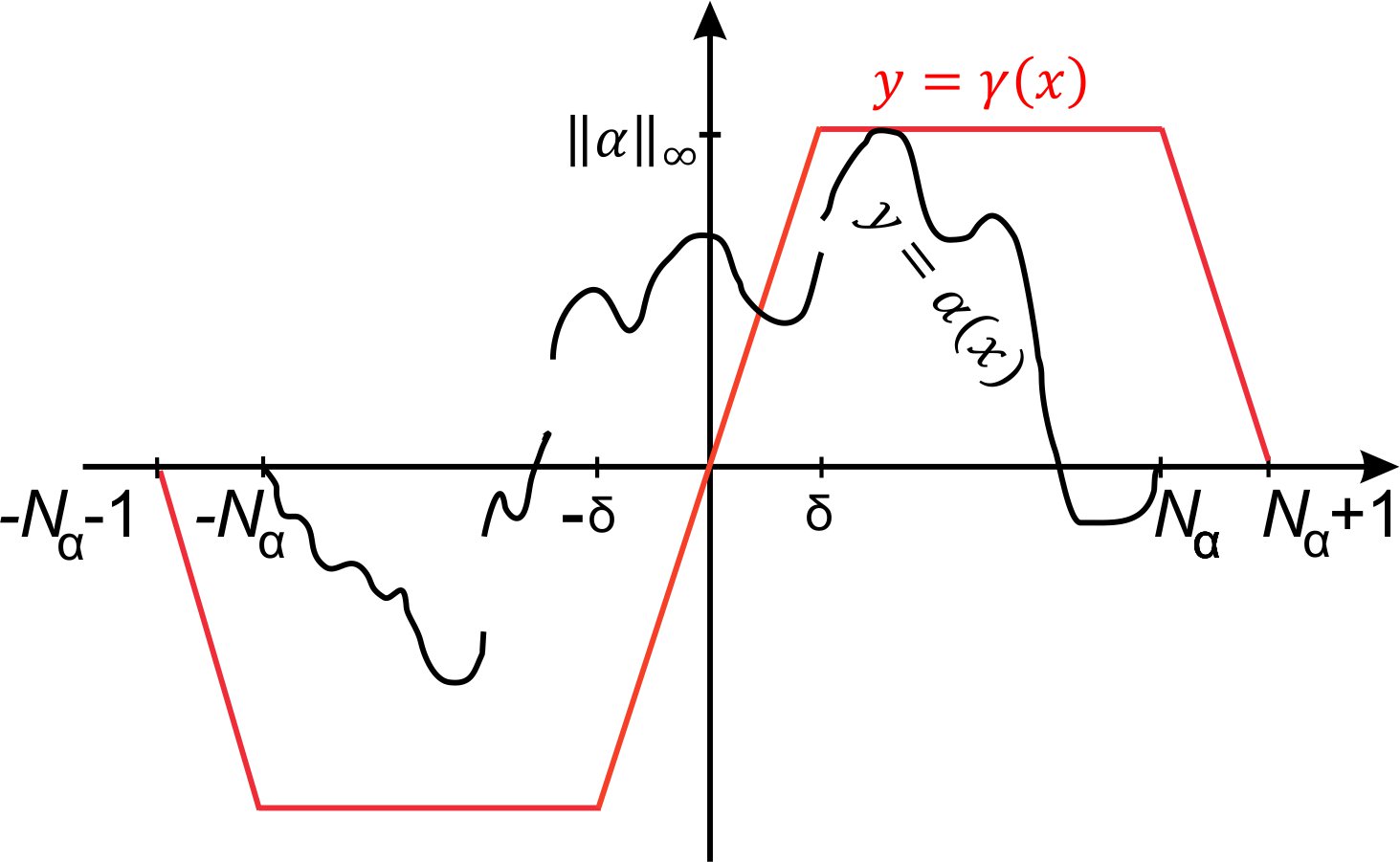}\\
 \caption{}
 \label{Fig. 1}
\end{figure}
Such function is clearly bounded, Lipschitz continuous and integrable. \\
Moreover, since by construction
$(\gamma(u)-\alpha(u))u \geq 0$ for any $|u|\geq \delta$, $u \mapsto \displaystyle \frac{s''(u)}{s'(u)}u+1$ is a.a. bounded from below  by 1 on that domain. \\
Inside of the interval $[-\delta, +\delta]$, since $\gamma(u)u \geq 0$, one has:
\begin{equation}\label{eq_ineq_posit_delta}
2\frac{\gamma(u)-\alpha(u)}{\sigma^2(u)}u+1\geq  -2\frac{ \alpha(u)}{\sigma^2(u)}u+1
\geq -2  \frac{\|\alpha\|_{\infty}}{c_\sigma}\delta+1.
\end{equation}
Choose $\delta = \displaystyle \frac  {c_\sigma}{4\|\alpha\|_{\infty}}$; one then obtains that $u \mapsto \displaystyle \frac{s''(u)}{s'(u)}u+1$ is bounded from below  by $1/2$ on $[-\delta, +\delta]$.
To summarize, we were able to construct a function $\gamma$ such that uniformly $\displaystyle (\tilde{id})'\geq 1/2$.

- If assumption ($\mathcal{A}_3'$) is fulfilled, there exists a bounded integrable Lipschitz
continuous function $g$ such that $g(u) >|\alpha(u)|, u\in\R$. Without loss of generality we may assume that $g$ is an even function.
In this case, set as above $\delta := \displaystyle \frac  {c_\sigma}{4\|\alpha\|_{\infty}}$ and define the odd function $\gamma$ as follows (see Figure \ref{Fig. 2}):
\begin{equation} \label{eq:gamma3'}
\gamma (u) =\left\{
\begin{aligned}
&g(\delta) \frac{u}{\delta}, & \ & u\in[0,\delta],\\
&g(u), & \ & u\in[\delta,+\infty),\\
&-\gamma(-u), & \ & u\in\R_-.
\end{aligned}\right.
\end{equation}
\begin{figure}[h]
   \includegraphics[width=12 cm]{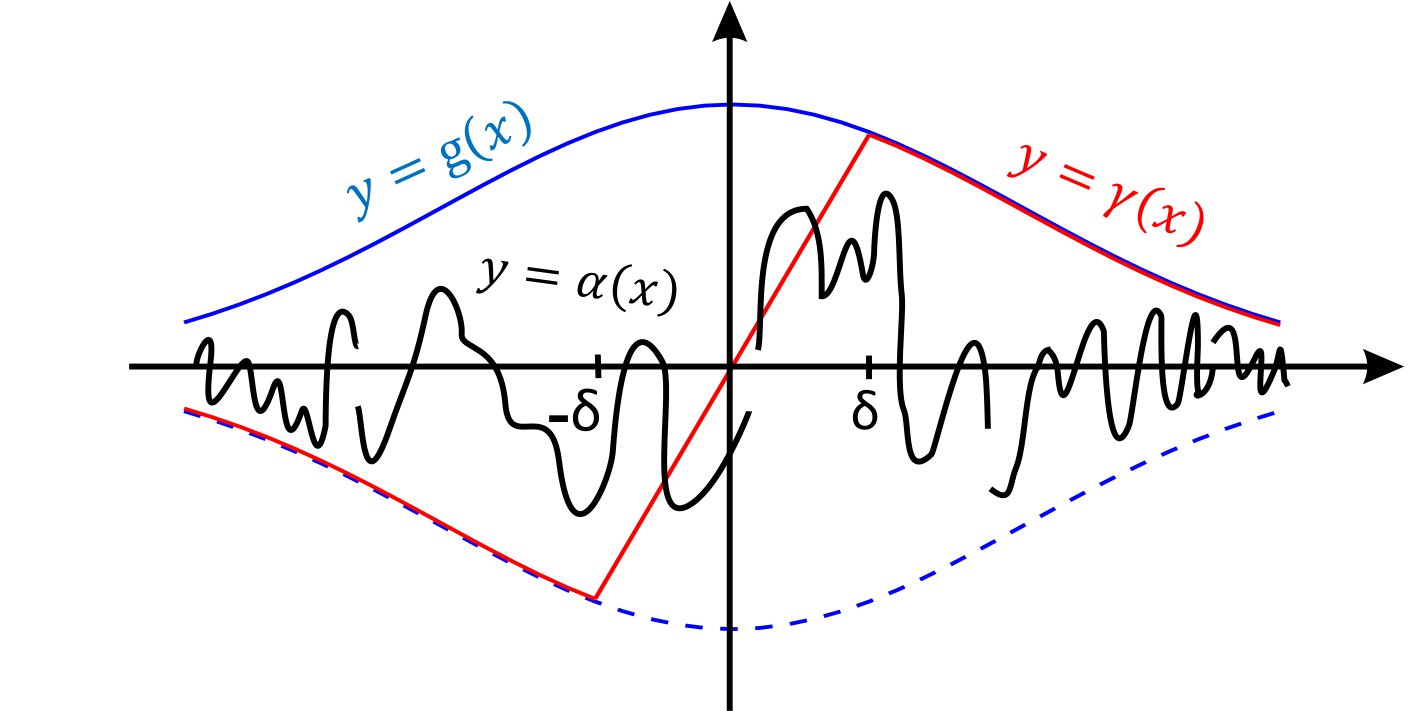}\\
 \caption{}
 \label{Fig. 2}
\end{figure}
By the same argumentation as in the first case, the function $(\tilde{id})'$ is
bounded from below by $1/2$.\\
\emph{Last steps of the proof of Theorem \ref{Theorem_main}}.\\
Applying now Proposition \ref{Prop1} to the process $(\tilde X_t)_{t\geq 0}$,
 thanks to \eqref{eq:Dtildeid}, one gets that for
 \mbox{$\displaystyle \lambda > 2(L_{\tilde \beta} + L_{\tilde \gamma})$}, the following a.s. synchronization holds
\begin{equation}\label{eq:synchXtilde}
\forall x,y\in \R, \quad  \quad \lim_{t\to +\infty}|\tilde X_t^y-\tilde X_t^x| \, e^{c t}= 0 \quad \mbox{ a.s. }
\end{equation}
for any $c < c_\lambda:= \displaystyle \frac{\lambda}{2} - L_{\tilde \beta} - L_{\tilde \gamma}\leq D_{\tilde b}$.\\
\vspace{3mm}
To deduce the  a.s. synchronization of the process $(X_t)_{t\geq 0}$ from
\eqref{eq:synchXtilde} we use the Lipschitz continuity of the function $s^{-1}$.
The exponential rate of convergence for both processes is then  identical.\\
\\
Hence, we  may select
\begin{equation}\label{eq:synch_lambda0}
\lambda_0:= 2(L_{\tilde \beta} + L_{\tilde \gamma})\ \mbox{ and }\ \ c_\lambda:= \frac{\lambda}{2} - L_{\tilde \beta} - L_{\tilde \gamma}.
\end{equation}
\\
We now compute an explicit upper bound for $ L_{\tilde \beta} + L_{\tilde \gamma} $
 using only the parameters of the SDE, and not $\gamma$.\\

- If assumption ($\mathcal{A}_3$) holds, one chooses $\gamma$ as in \eqref{eq:gamma3}. Therefore, by \eqref{eq_scale_deriv_bound}, one has
$$
L_s\leq \exp\left(8 \frac{\|\alpha\|_\infty (N_\alpha +1)}{c_\sigma^2}\right)
$$
and by \eqref{eq_scale_deriv_Lip},
$$
L_{s'} \leq \frac{4 \|\alpha\|_\infty }{c_\sigma} L_s
\leq \frac{4 \|\alpha\|_\infty }{c_\sigma}\exp\left(8 \frac{\|\alpha\|_\infty (N_\alpha +1)}{c_\sigma^2}\right) .
$$
Therefore, using the definition \eqref{eq_tilde_beta},
$$
L_{\tilde \beta} \leq \left(L_\beta+\|\beta\|_{N_\alpha+1}\frac{4 \|\alpha\|_\infty }{c_\sigma}\right)\exp\left(16 \frac{\|\alpha\|_\infty (N_\alpha +1)}{c_\sigma^2}\right)
$$
and
\begin{eqnarray*}
L_{\tilde \gamma} &\leq& \displaystyle
\left(L_\gamma+\|\gamma\|_\infty \frac{4 \|\alpha\|_\infty }{c_\sigma}\right) \exp\left(16 \frac{\|\alpha\|_\infty (N_\alpha +1)}{c_\sigma^2}\right)\\
&\leq & \|\alpha\|_\infty \left(1 + \frac{4 \|\alpha\|_\infty }{c_\sigma}\right) \exp\left(16 \frac{\|\alpha\|_\infty (N_\alpha +1)}{c_\sigma^2}\right).
 \end{eqnarray*}
 So
\begin{multline}\label{eq:clambda}
L_{\tilde \beta} + L_{\tilde \gamma}\\ \leq
\left(\|\alpha\|_\infty + L_\beta + (\|\alpha\|_\infty +\|\beta\|_{N_\alpha+1}) \frac{4 \|\alpha\|_\infty }{c_\sigma}\right)\exp\left(16 \frac{\|\alpha\|_\infty (N_\alpha +1)}{c_\sigma^2}\right)
\end{multline}

- If assumption ($\mathcal{A}_3'$) holds, one chooses $\gamma$ as in \eqref{eq:gamma3'}. By \eqref{eq_scale_deriv_bound}, one has
$L_s\leq \exp\left( 4 \frac{\|g\|_1}{c_\sigma^2}\right) $
and by \eqref{eq_scale_deriv_Lip},
$$
L_{s'} \leq \frac{4 \|g\|_\infty }{c_\sigma} L_s
\leq \frac{4 \|g\|_\infty }{c_\sigma}\exp\left( 4 \frac{\|g\|_1}{c_\sigma^2}\right)  .
$$
Therefore, using the definition \eqref{eq_tilde_sigma},
$$
L_{\tilde \beta} \leq \left(L_\beta+\|\beta\|_\infty \frac{4 \|g\|_\infty }{c_\sigma}\right)
\exp\left(8 \frac{\|g\|_1}{c_\sigma^2}\right)
$$
and
\begin{eqnarray*}
L_{\tilde \gamma} &\leq& \displaystyle
\left(L_\gamma+\|\gamma\|_\infty \frac{4 \|g\|_\infty }{c_\sigma}\right) \exp\left(8 \frac{\|g\|_1}{c_\sigma^2}\right)  \\
&\leq & \left(L_g+\frac{8 \|g\|^2_\infty }{c_\sigma}\right) \exp\left(8 \frac{\|g\|_1}{c_\sigma^2}\right) .
 \end{eqnarray*}
 So in that case,
\begin{equation}\label{eq:clambda3'}
L_{\tilde \beta} + L_{\tilde \gamma} \leq
\left(L_\beta + L_g +  (\|\beta\|_\infty + 2 \|g\|^2_\infty )\frac{4 \|g\|_\infty }{c_\sigma} \right)
\exp\left(8 \frac{\|g\|_1}{c_\sigma^2}\right) .
\end{equation}

The $L_p$-synchronization of  $(X_t)_{t\geq 0}$ is a direct consequence from the fact that $(\tilde X_t)_{t\geq 0}$    satisfies the $L_p$-bounds \eqref{eq_stability_dissip_Lp}: take $C=L_s $ and $c_{\lambda,p}= p c_\lambda -\frac{p(p-1)}{2}L_{\tilde \sigma}^2$.  Indeed, under assumption ($\mathcal{A}_3$),
$$
L_{\tilde \sigma} \leq \left(L_\sigma+\|\sigma\|_{N_\alpha+1}\frac{4 \|\alpha\|_\infty }{c_\sigma}\right)\exp\left(16 \frac{\|\alpha\|_\infty (N_\alpha +1)}{c_\sigma^2}\right)
$$
and under assumption ($\mathcal{A}_3'$)
$$
L_{\tilde \sigma} \leq \left(L_\beta+\|\sigma\|_\infty \frac{4 \|g\|_\infty }{c_\sigma}\right)
\exp\left(8 \frac{\|g\|_1}{c_\sigma^2}\right) .
$$
The constant $c_{\lambda,p}$ can also be estimated explicitly as function of the parameters of the SDE. This completes the proof.
\end{proof}
\section*{Acknowledgements}
The authors would like to warmly thank M. Scheutzow  and S. Mazzonetto for fruitful discussions on this topic. It is also their pleasure to thank an anonymous referee for drawing their attention to the
cited paper of M. Scheutzow and S. Schulze, which allowed a significant improvement of a first version of Proposition 1.


\end{document}